\documentclass[12pt]{article}

\usepackage{amsthm,amssymb,amsmath}
\usepackage[a4paper]{geometry}
\author{Jean-Paul Allouche\thanks{The author was partially supported by the ANR
project ``FAN'' (Fractals et Num\'eration), ANR-12-IS01-0002.} \\
CNRS, Institut de Math\'ematiques de Jussieu-PRG \\
Universit\'e Pierre et Marie Curie, Case 247 \\
4 Place Jussieu \\
F-75252 Paris Cedex 05 France \\
{\tt jean-paul.allouche@imj-prg.fr}
\and
Jonathan Sondow \\
209 West 97th Street \\
New York \\
NY 10025, USA \\
{\tt jsondow@alumni.princeton.edu}
}

\title{Summation of rational series twisted by 
strongly $B$-multiplicative coefficients}

\date{ }

\theoremstyle{plain}
\newtheorem{theorem}{Theorem}
\newtheorem{lemma}[theorem]{Lemma}
\newtheorem{corollary}[theorem]{Corollary}

\theoremstyle{definition}
\newtheorem{definition}[theorem]{Definition}
\newtheorem{example}[theorem]{Example}

\theoremstyle{remark}

\begin{document}

\maketitle

\begin{abstract}
We evaluate in closed form series of the type $\sum u(n) R(n)$, with $(u(n))_n$
a strongly $B$-multiplicative sequence and $R(n)$ a (well-chosen) rational
function. A typical example is:
$$
\sum_{n \geq 1} (-1)^{s_2(n)} \frac{4n+1}{2n(2n+1)(2n+2)} = -\frac{1}{4}
$$
where $s_2(n)$ is the sum of the binary digits of the integer $n$.
Furthermore closed formulas for series involving automatic sequences that are not strongly 
$B$-multipli\-ca\-tive, such as the regular paperfolding and Golay-Shapiro-Rudin sequences, 
are obtained; for example, for integer $d \geq 0$:
$$
\sum_{n \geq 0} \frac{v(n)}{(n+1)^{2d+1}} = \frac{\pi^{2d+1} |E_{2d}|}{(2^{2d+2}-2)(2d)!}
$$
where $(v(n))_n$ is the $\pm 1$ regular paperfolding sequence and $E_{2d}$ is an Euler number. 

\bigskip


\bigskip

\end{abstract}

\section{Introduction}

The problem of evaluating a series $\sum_n R(n)$ where $R$ is a rational function
with integer coefficients is classical: think of the values of the Riemann $\zeta$ 
function at integers. Such sums can also be ``twisted'', usually by a character
(think of the $L$-functions), or by the usual arithmetic functions (e.g., the M\"obius
function $\mu$). 

\bigskip

\noindent Another possibility is to twist such sums by sequences related to
the digits of $n$ in some integer base. Examples can be found in \cite{AllSha-lnm} 
with, in particular, series $\sum \frac{u(n)}{n(n+1)}$, and in \cite{AllShaSon}
with, in particular, series $\sum \frac{u(n)}{2n(2n+1)}$ (also see \cite{PP2010}): 
in both cases $u(n)$ counts the number of occurrences of a given block of digits in 
the $B$-ary expansion of the integer $n$, or is equal to $s_B(n)$, the sum of the 
$B$-ary digits of the integer $n$ ($B$ being an integer $\geq 2$). Two emblematic 
examples are (see \cite[Problem B5, p.\ 682]{Putnam} and \cite{Shallit, AllSha-lnm} 
for the first one, and \cite{Sondow2, AllShaSon} for the second one):
$$
\sum_{n \geq 1} \frac{s_B(n)}{n(n+1)} = \frac{B}{B-1}
\ \ \mbox{\rm and} \ \
\sum_{n \geq 1} \frac{s_2(n)}{2n(2n+1)} = \frac{\gamma + \log\frac{4}{\pi}}{2}
$$
where $\gamma$ is the Euler-Mascheroni constant. 

\bigskip

\noindent Similarly one can try to evaluate infinite products $\prod_n R(n)$, where
$R(n)$ is a rational function, as well as twisted such products $\prod_n R(n)^{u(n)}$, 
where the sequence $(u(n))_{n \geq 0}$ is related to the digits of $n$ in some integer 
base. An example can be found in \cite{Allouche2013} (also see \cite{Rivoal2005} for the 
original problem):
$$
\prod_{n \geq 1} \left(\frac{(4n+2)(4n+2)}{(4n+1)(4n+3)}\right)^{2z(n)} = \frac{4}{\pi}
$$
where $z(n)$ is the sum of the number of $0$'s and the number of $1$'s in the binary
expansion of $n$, i.e., the length of this expansion.
Other examples can be found in \cite{AllSha-jlms}, e.g.,
$$
\prod_{n \geq 0} \left(\frac{(4n+2)(8n+7)(8n+3)(16n+10)}{(4n+3)(8n+6)(8n+2)(16n+11)}\right)^{u(n)}
= \frac{1}{\sqrt{2}}
$$
where $u(n)=(-1)^{a(n)}$ and $a(n)$ is equal to the number of blocks $1010$ occurring in the 
binary expansion of $n$. The products studied in \cite{AllSha-jlms} (also see references therein)
are of the form $\prod_n R(n)^{(-1)^{a(n)}}$ where $R(n)$ is a (well-chosen) rational function with 
integer coefficients, and $a(n)$ counts the number of occurrences of a given block of digits in the 
$B$-ary expansion of the integer $n$. The case where $a(n)$ counts the number of $1$'s occurring in 
the binary expansion of $n$ is nothing but the case $a(n)=s_2(n)$. If $a(n)=s_B(n)$, the sequence 
$((-1)^{a(n)})_{n \geq 0}$ is strongly $B$-multiplicative: the more general evaluation of the product 
$\prod_n R(n)^{u(n)}$ where $(u(n))_{n \geq 0}$ is a strongly $B$-multiplicative sequence, is addressed 
in \cite{AllSon} (also see \cite{Sondow1}). 
Recall that a {\it strongly $B$-multiplicative sequence} $(u(n))_{n \geq 0}$ satisfies $u(0) = 0$, and 
$u(Bn+j) = u(n) u(j)$ for all $j \in [0, B-1]$ and all $n \geq 0$. In particular, $(u(n))_{n \geq 0}$
is $B$-regular (or even $B$-automatic if it takes only finitely many values): recall that a sequence
$(u(n))_{n \geq 0}$ is called {\it $B$-automatic\,} if its {\it $B$-kernel}, i.e., the set of subsequences
$\{(u(B^a n+r))_{n \geq 0} \mid \ a \geq 0, \ 0 \leq r \leq B^a -1\}$, is finite; a sequence 
$(u(n))_{n \geq 0}$ with values in ${\mathbb Z}$ is called {\it $B$-regular} if the ${\mathbb Z}$-module
spanned by its $B$-kernel has finite type (for more on these notions, see, e.g., \cite{AS}).

\bigskip

\noindent
Since $\log \prod_n R(n)^{u(n)} = \sum_n u(n) \log R(n)$, it is natural to look at
``simpler'' series of the form $\sum_n u(n) R(n)$ with $R$ and $u$ as previously. 
All the examples above involve sequences $(u(n))_{n \geq 0}$ that are $B$-regular or even
$B$-automatic. Unfortunately we were not able to address the general case where 
$(u(n))_{n \geq 0}$ is any $B$-regular or any $B$-automatic sequence.
The purpose of the present paper is to study the special case where, as in \cite{AllSon}, the
sequence $u(n)$ is strongly $B$-multiplicative and $R(n)$ is a well-chosen rational function.
The paper can thus be seen as a companion paper to \cite{AllSon}. We will end with the 
evaluation of similar series where $(u(n))_{n \geq 0}$ is the regular paperfolding sequence 
or the Golay-Shapiro-Rudin sequence.
 
\section{Preliminary definitions and results}

This section quickly recalls definitions and results from \cite{AllSon}. 

\begin{definition}\label{defmult}
Let $B \geq 2$ be an integer.  A sequence of complex numbers
$(u(n))_{n \geq 0}$ is {\em strongly $B$-multiplicative} if $u(0)=1$
and, for all $n \geq 0$ and all $k \in \{0, 1, \ldots, B-1\}$,
$$
u(Bn + k) = u(n) u(k).
$$
\end{definition}

\begin{example}\label{example-sB}
Let $B \geq 2$ be an integer and $s_B(n)$ be the sum of the $B$-ary digits of $n$. Then
for every complex number $a \neq 0$ the sequence $(a^{s_B(n)})_{n \geq 0}$ is strongly
$B$-multiplicative. This sequence is $B$-regular (see the introduction); it is $B$-automatic
if and only if $a$ is a root of unity.
\end{example}

The following lemma is a variation of Lemma~1 in \cite{AllSon}.

\begin{lemma}\label{conv}
Let $B > 1$ be an integer. Let $(u(n))_{n \geq 0}$ be a strongly $B$-multiplicative sequence 
of complex numbers different from the sequence $(1,0,0,\ldots)$. We suppose that $|u(n)| \leq 1$ 
for all $n \geq 0$ and that $|\sum_{0 \leq k < B} u(k)| < B$. Let $f$ be a map from the set of
nonnegative integers to the set of complex numbers such that $|f(n+1) - f(n)| = {\mathcal O}(n^{-2})$. 
Then the series $\sum_{n \geq 0} u(n) f(n)$ is convergent.
\end{lemma}

\begin{proof} Use \cite[Lemma 1]{AllSon} to get the upper bound 
$|\sum_{0 \leq n < N} u(n)| < C N^{\alpha}$ for some positive constant $C$ 
and some real number $\alpha$ in $(0, 1)$. Then use summation by parts. 
\end{proof}

\section{Main results}

We state in this section some basic identities as well as first applications and examples.
First we define $\delta_k$, a special case of the Kronecker delta:
$$
\delta_k = 
\begin{cases}
1 \ &\mbox{\rm if} \ k = 0, \\
0 \ &\mbox{\rm otherwise}.
\end{cases}
$$

\begin{theorem}\label{first}
Let $B > 1$ be an integer. Let $(u(n))_{n \geq 0}$ be a strongly $B$-multiplicative sequence, 
and let $f$ be a map from the nonnegative integers to the complex numbers, such that 
$(u(n))_{n \geq 0}$ and $f$ satisfy the conditions of Lemma~\ref{conv}. Define the series 
$S_1(k, B, u, f)$, for $k = 0, 1, \ldots, B-1$, by
$$
S(k, B, u, f) := \sum_{n \geq 0} u(n) f(Bn+k).
$$
Then the following linear relations hold:
$$
\sum_{n \geq 0} u(n) f(n) = \sum_{0 \leq k \leq B-1} u(k) S(k, B, u, f) 
$$
and
$$
\sum_{n \geq 0} u(n) \sum_{0 \leq k \leq B-1} f(Bn+k) = \sum_{0 \leq k \leq B-1} S(k, B, u, f).
$$

\medskip

\noindent
In particular, define the series $S_1(k, B, u)$ and $S_2(k, B, u)$, for
$k = 0, 1, \ldots, B - 1$, by
$$
S_1(k, B, u) := \sum_{n \geq \delta_k} \frac{u(n)}{Bn+k}
\ \mbox{and} \
S_2(k, B, u) := \sum_{n \geq \delta_k} \frac{u(n)}{(Bn+k)(Bn+k+1)}\cdot
$$
Then the following linear relations hold:
$$
(B-1) S_1(0, B, u) - \sum_{1 \leq k \leq B-1} u(k) S_1(k, B, u) = 0
$$
and
$$
\sum_{0 \leq k \leq B-1} (B - u(k)) S_2(k, B, u)= B - 1.
$$
\end{theorem}

\begin{proof} It follows from Lemma~\ref{conv} that all the series in the theorem
converge. To prove the first relation, we split $\sum_{n \geq 0} u(n)f(n)$, obtaining
$$
\begin{array}{lll}
\displaystyle\sum_{n \geq 0} u(n)f(n) &=& 
\displaystyle\sum_{0 \leq k \leq B-1} \sum_{n \geq 0} u(Bn+k)f(Bn+k)
= \sum_{0 \leq k \leq B-1} \sum_{n \geq 0} u(n)u(k)f(Bn+k) \\
&=& \displaystyle\sum_{0 \leq k \leq B-1} u(k) \sum_{n \geq 0} u(n) f(Bn+k)
= \sum_{0 \leq k \leq B-1} u(k) S(k, B, u, f).
\end{array}
$$
To prove the second relation, we write
$$
\sum_{n \geq 0} u(n) \sum_{0 \leq k \leq B-1} f(Bn+k)
= \sum_{0 \leq k \leq B-1} \sum_{n \geq 0} u(n) f(Bn+k) 
= \sum_{0 \leq k \leq B-1} S(k, B, u, f). 
$$
To prove the last part of the theorem, we make two choices for $f$.
First we take $f$ defined by $f(n) = 1/n$ for $n \neq 0$ and $f(0) = 0$.
Then we take $f(n) = 1/n(n+1)$ if $n \neq 0$ and $f(0) = 0$.  
\end{proof}

\bigskip

\noindent
{\bf Remark.} \ The formula $S_2(k,B,u) = S_1(k,B,u) - (S_1(k+1,B,u) - \delta_k)$ 
$(0 \leq k \leq B - 2)$ holds. Nevertheless, the last two relations in Theorem~\ref{first} 
are independent, because $S_2(B - 1, B, u)$ cannot be expressed in terms of the 
$S_1(k, B, u)$ for $k = 0, 1, \ldots, B - 1$.

\begin{corollary}\label{cor-gen}
If $(u(n))_{n \geq 0}$ is a strongly $B$-multiplicative sequence satisfying the conditions of
Lemma~\ref{conv}, then
$$
\sum_{n \geq 1} u(n) \sum_{1 \leq k \leq B-1} \left(\frac{1}{Bn} - \frac{u(k)}{Bn+k}\right)
= \sum_{1 \leq k \leq B-1} \frac{u(k)}{k}
$$
and
$$
\sum_{n \geq 1} u(n) \sum_{0 \leq k \leq B-1} \frac{B-u(k)}{(Bn+k)(Bn+k+1)} 
= \sum_{1 \leq k \leq B-1} \frac{u(k)}{k(k+1)}\cdot
$$
\end{corollary}

\begin{proof} This follows from the last part of Theorem~\ref{first} 
by substitution and manipulation. 
\end{proof}

\bigskip

Recall that the $n$th harmonic number $H_n$ and the $n$th alternating harmonic number 
$H_n^*$ are defined by
$$
H_n := \sum_{1 \leq k \leq n} \frac{1}{k} \ \mbox{\rm and} \ 
H_n^* := \sum_{1 \leq k \leq n} \frac{(-1)^{k-1}}{k}\cdot
$$

\begin{corollary}
If $N_{j,B}(n)$ is the number of occurrences of the digit $j \in \{0,1,\ldots,B-1\}$ in the 
$B$-ary expansion of $n$, then the following summations hold when $j \neq 0$:
$$
\sum_{n \geq 1} (-1)^{N_{j,B}(n)} \left(\frac{2}{Bn+j} + 
\frac{1}{Bn} \sum_{1 \leq k \leq B - 1} \frac{k}{Bn+k}\right) 
= H_{B-1} - \frac{2}{j}
$$
and
$$
\sum_{n \geq 1} (-1)^{N_{j,B}(n)} \left(\frac{B-1}{n(n+1)} + \frac{2B}{(Bn+j)(Bn+j+1)}\right)
= B - 1 - \frac{2B}{j(j+1)}\cdot
$$
\end{corollary}

\begin{proof} It is not hard to see that, if $j \neq 0$, we can apply the last part of 
Theorem~\ref{first} to the sequence $u(n) := (-1)^{N_{j,B}(n)}$. Using Corollary~\ref{cor-gen} 
and the fact that $N_{j,B}(k) = \delta_{k,j}$ when $0 \leq k < B$, the result follows. 
\end{proof}

\begin{example}\label{ex7} Taking $B = 2$ and $j = 1$, we get
$$
\sum_{n \geq 1} (-1)^{N_{1,2}(n)} \frac{4n+1}{2n(2n+1)} = -1
$$
and
$$
\sum_{n \geq 1} (-1)^{N_{1,2}(n)} \frac{4n+1}{2n(2n+1)(2n+2)} = -\frac{1}{4}\cdot
$$
Subtracting the second equation from the first, we multiply by $4$ and obtain
$$
\sum_{n \geq 1} (-1)^{N_{1,2}(n)} \frac{4n+1}{n(n+1)} = -3.
$$
With $B = 3$ and $j = 1$ we get
$$
\sum_{n \geq 1} (-1)^{N_{1,3}(n)} \frac{18n^2 + 21n + 4}{3n(3n+1)(3n+2)} = - \frac{1}{2}
$$
and
$$
\sum_{n \geq 1} (-1)^{N_{1,3}(n)} \frac{6n^2 + 6n + 1}{3n(3n+1)(3n+2)(3n+3)} = - \frac{1}{36}\cdot
$$
\end{example}

\begin{corollary}
If $s_B(n)$ is the sum of the $B$-ary digits of $n$, then
$$
\sum_{n \geq 1} (-1)^{s_B(n)} \sum_{1 \leq k \leq B-1} \left(\frac{1}{Bn} - \frac{(-1)^k}{Bn+k}\right)
= - H_{B-1}^*
$$
and
$$
\sum_{n \geq 1} (-1)^{s_B(n)} \sum_{0 \leq k \leq B-1} \frac{B - (-1)^k}{(Bn+k)(Bn+k+1)}
= 1 + \frac{(-1)^B}{B} - 2 H_{B-1}^*.
$$
\end{corollary}

\begin{proof} Setting $u(n):= (-1)^{s_B(n)}$, it is not hard to see that $u(2n +1) = - u(2n)$ for all 
$n \geq 0$. (Hint: look at the cases $B$ even and $B$ odd separately.) It follows that 
$(u(n))_{n \geq 0}$ satisfies the conditions of Lemma~\ref{conv}. Noting that $u(k) = (-1)^k$ 
when $0 \leq k < B$, the result follows from Corollary~\ref{cor-gen}.
\end{proof}

\begin{example}
Taking $B = 2$ or $3$ gives the same pair of series as those with that value of $B$ in 
Example~\ref{ex7}, since $s_2(n) = N_{1,2}(n)$ and $s_3(n) = N_{1,3}(n) + 2 N_{2,3}(n)$. 
(We can also replace $s_3(n)$ with $n$, as $(-1)^{s_B(n)} = (-1)^n$ when $B$ is odd.) 
With $B = 4$ we get
$$
\sum_{n \geq 1} (-1)^{s_4(n)} \frac{128n^3 + 176n^2 + 76n + 9}{4n(4n+1)(4n+2)(4n+3)} = 
- \frac{5}{12}
$$
and
$$
\sum_{n \geq 1} (-1)^{s_4(n)} \frac{128n^3 + 184n^2 + 80n + 9}{4n(4n+1)(4n+2)(4n+3)(4n+4)} =
- \frac{5}{12}\cdot
$$
\end{example}

\section{More examples}

Using Corollary~\ref{cor-gen} with sequences $(u(n))_{n \geq 0}$ taking complex
values yields other examples of sums of series.

\begin{example}~\label{complex1}
 We may let $u(n) := i^{s_2(n)}$ in Corollary~\ref{cor-gen}. This gives the two summations
$$
\sum_{n \geq 1} \left(\frac{i^{s_2(n)}}{2n} - \frac{i^{s_2(n)+1}}{2n+1}\right) = i 
= \sum_{n \geq 1} \frac{i^{s_2(n)}(3n+1) - i^{s_2(n)+1}n}{n(n+1)(2n+1)},
$$
and by taking the imaginary and real parts we obtain the following result:

\medskip

{\it If $\chi$ is the non-principal Dirichlet character modulo $4$, defined by
$$
\chi(n) :=
\begin{cases}
+1 \ &\mbox{\rm if} \ n \equiv 1 \bmod 4, \\
-1 \ &\mbox{\rm if} \ n \equiv 3 \bmod 4, \\
\ \ 0  \ &\mbox{\rm otherwise},
\end{cases}
$$
then
$$ 
\sum_{n \geq 1} \left(\frac{\chi(s_2(n))}{2n} - \frac{\chi(s_2(n)+1)}{2n+1}\right) = 1
= \sum_{n \geq 1} \frac{(3n+1)\chi(s_2(n)) - n \chi(s_2(n)+1)}{n(n+1)(2n+1)}
$$
and
$$ 
\sum_{n \geq 1} \left(\frac{\chi(s_2(n)+1)}{2n} - \frac{\chi(s_2(n)+2)}{2n+1}\right) = 0
= \sum_{n \geq 1} \frac{(3n+1)\chi(s_2(n)+1) - n \chi(s_2(n)+2)}{n(n+1)(2n+1)}\cdot
$$
}
\end{example}

\begin{example}
Generalizing Example~\ref{complex1} by replacing $i^{s_2(n)}$ with 
$e^{2i \pi s_2(n)/d}$, for integer $d \geq 2$, is straightforward, yielding 
the following summations (Example~\ref{complex1} is another formulation for the
case $d=4$):
$$ 
\sum_{n \geq 1} \left(\frac{\sin\frac{2\pi s_2(n)}{d}}{2n} 
- \frac{\sin\frac{2\pi (s_2(n)+1)}{d}}{2n+1}\right) 
= \sin\frac{2\pi}{d} 
= \sum_{n \geq 1} \frac{(3n+1)\sin\frac{2\pi s_2(n)}{d} - n \sin\frac{2\pi (s_2(n)+1)}{d}}
{n(n+1)(2n+1)}
$$ 
and 
$$
\sum_{n \geq 1} \left(\frac{\cos\frac{2\pi s_2(n)}{d}}{2n} 
- \frac{\cos\frac{2\pi (s_2(n)+1)}{d}}{2n+1}\right) 
= \cos\frac{2\pi}{d} 
= \sum_{n \geq 1} \frac{(3n+1)\cos\frac{2\pi s_2(n)}{d} - n \cos\frac{2\pi (s_2(n)+1)}{d}}
{n(n+1)(2n+1)}\cdot
$$ 
\end{example}

\section{The paperfolding and Golay-Shapiro-Rudin sequences}

The results above involve sums $\sum u(n) R(n)$ where $(u(n))_{n \geq 0}$ is a strongly 
$B$-multiplicative sequence, which, in all of our examples except 
Example~\ref{example-sB} with alpha not a root of unity, happens to take only finitely
many values. This implies that $(u(n))_{n \geq 0}$ is {\it $B$-automatic} (see the 
introduction). One can then ask about more general sums $\sum u(n) R(n)$ where the sequence 
$(u(n))_{n \geq 0}$ is $B$-automatic. We give two cases where such series can be summed.

\begin{theorem}\label{paperfold}
Let $(v(n))_{n \geq 0}$ be the {\em regular paperfolding sequence}. Its first few terms are given by
(replacing $+1$ by $+$ and $-1$ by $-$)
$$
(v(n))_{n \geq 0} = + \ + \ - \ + \ + \ - \ - \ \ldots;
$$ 
it can be defined by: $v(2n)=(-1)^n$ and $v(2n+1) = v(n)$ for all $n \geq 0$. 
Then, for all integers $d \geq 0$, we have the relation
$$
\sum_{n \geq 0} \frac{v(n)}{(n+1)^{2d+1}} = \frac{\pi^{2d+1} |E_{2d}|}{(2^{2d+2}-2)(2d)!}
$$
where the $E_{2d}$'s are the Euler numbers defined by:
$$
\frac{1}{\cosh t} = \sum_{n \geq 0} \frac{E_{2n}}{(2n)!} t^{2n} \ \mbox{\rm for} \ |t| < \frac{\pi}{2}\cdot
$$
\end{theorem}

\begin{proof} First note that the series $\sum_{n \geq 0} \frac{v(n)}{(n+1)^s}$ converges 
for $\Re(s) > 0$: use the inequality $|\sum_{n < N} v(n)| = O(\log N)$ (see, e.g., 
\cite[Exercise~28, p.~206]{AS}) and summation by parts; note that the sequence $(R_n)_{n \geq 1}$ 
in \cite[Exercise~28, p.~206]{AS} is equal to the sequence $(v(n))_{n \geq 0}$ here. 
Now, Exercise~27 in \cite[p.\ 205--206]{AS} asks to prove, for all complex numbers $s$ with 
$\Re(s) > 0$, the equality (again with slightly different notation)
$$
\sum_{n \geq 0} \frac{v(n)}{(n+1)^s} 
= \frac{2^s}{2^s - 1} \sum_{n \geq 0} \frac{(-1)^n}{(2n+1)^s}\cdot
$$ 
This can be easily done by splitting the sum on the left into even and odd indexes. 
Recalling that the Dirichlet beta function is defined by 
$\beta(s) = \sum_{n \geq 0} \frac{(-1)^n}{(2n+1)^s}$ for $\Re(s) > 0$, we thus have, 
for any nonnegative integer $d$, 
$$
\sum_{n \geq 0} \frac{v(n)}{(n+1)^{2d+1}}  = \frac{2^{2d+1}}{2^{2d+1} - 1} \beta(2d+1).
$$ 
But, when $s$ is an odd integer, the value of $\beta(s)$ can be expressed as a rational multiple
of $\pi$ (see, e.g., \cite[23.2.22, p.\ 807]{Abr-Ste}):
$$
\beta(2d+1) = \frac{(\pi/2)^{2d+1}}{2 (2d)!} |E_{2d}|.  \qedhere
$$
\end{proof}

\begin{example}
Taking $d=0$ in Theorem~\ref{paperfold} yields a result due to F. von Haeseler
(see \cite[Exercise~27, p.\ 205--206]{AS})
$$
\sum_{n \geq 0} \frac{v(n)}{n+1} = \frac{\pi}{2}\cdot
$$
\end{example}

\bigskip

\noindent
{\bf Remark.} \
The paperfolding sequence happens to be related to the Jacobi-Kronecker symbol
(see, e.g., \cite[p.~27--28]{Cohen}). Namely, as noted in \cite{oeis} for the
sequence A034947, the following identity holds
$$
v(n-1) = \left(\frac{-1}{n}\right) \ \mbox{\rm for} \ n \geq 1
$$
(denoting $R(n) := v(n-1)$ for $n \geq 1$, this is an easy consequence of the relations
$R(2n+1) = (-1)^n$ for all $n \geq 0$ and $R(2n) = R(n)$ for all $n \geq 1$).

\bigskip

The second result we give in this section involves the Golay-Shapiro-Rudin sequence.

\begin{theorem}\label{shap}
Let $(r(n))_{n \geq 0}$ be the $\pm 1$ {\em Golay-Shapiro-Rudin sequence}. This sequence can be defined 
by $r(n) = (-1)^{a(n)}$, where $a(n)$ is the number of possibly overlapping occurrences of 
the block $11$ in the binary expansion of $n$, so that (replacing $+1$ by $+$ and $-1$ by $-1$)
$$
(r(n))_{n \geq 0} = + \ + \ + \ - \ + \ + \ - \ + \ \ldots;
$$
alternatively it can be defined by
$$
r(0) = 1, \ \mbox{\rm and} \ r(2n) = r(n), \ r(2n+1) = (-1)^n r(n) \ \mbox{\rm for} \ n \geq 0.
$$
Let $R(n)$ be a function from the nonnegative integers to the complex numbers, 
such that $|R(n+1) - R(n)| = {\mathcal O}(n^{-2})$. Then we have the relation
$$
\sum_{n \geq 1} r(n) (R(n) - R(2n) + R(2n+1) - 2R(4n+1)) = R(1).
$$
\end{theorem}

\begin{proof} It is well known that $|\sum_{n < N} r(n)| < K \sqrt{n}$ for some positive constant 
$K$ (actually more is known; see, e.g., \cite[Theorem~3.3.2, p.\ 79]{AS} and the historical 
comments given in \cite[3.3, p.\ 121]{AS}). Thus, by summation by parts, the series 
$\sum_{n \geq 0} r(n) R(n)$ is convergent. Now we write
$$
\begin{array}{lll}
\displaystyle \sum_{n \geq 0} r(n) R(n) &=& 
\displaystyle \sum_{n \geq 0} r(2n) R(2n) + \sum_{n \geq 0} r(2n+1) R(2n+1) \\
&=& \displaystyle \sum_{n \geq 0} r(n) R(2n) + \sum_{n \geq 0} (-1)^n r(n) R(2n+1) \\
&=& \displaystyle \sum_{n \geq 0} r(n) R(2n) + \sum_{n \geq 0} r(2n) R(4n+1) 
- \sum_{n \geq 0} r(2n+1) R(4n+3) \\
&=& \displaystyle \sum_{n \geq 0} r(n) (R(2n) + R(4n+1)) -  \sum_{n \geq 0} r(2n+1) R(4n+3). \\ 
\end{array}
$$
Hence
$$
\begin{array}{lll}
\displaystyle \sum_{n \geq 0} r(n) (R(n) - R(2n) - R(4n+1)) &=&
-  \displaystyle \sum_{n \geq 0} r(2n+1) R(4n+3) \\
&=&
\displaystyle - (\sum_{n \geq 0} r(n) R(2n+1) - \sum_{n \geq 0} r(2n) R(4n+1)) \\
&=& \displaystyle - \sum_{n \geq 0} r(n) R(2n+1) + \sum_{n \geq 0} r(n) R(4n+1) \\
\end{array}
$$
where the penultimate equality is obtained by splitting the sum $\sum_{n \geq 0} r(n) R(2n+1)$
into even and odd indices. Thus, finally
$$
\sum_{n \geq 0} r(n) (R(n) - R(2n) + R(2n+1) - 2R(4n+1)) = 0,
$$
hence
$$
\sum_{n \geq 1} r(n) (R(n) - R(2n) + R(2n+1) - 2R(4n+1)) = R(1).  \qedhere
$$
\end{proof}

\begin{example}
Taking $R(n) = 1/n$ if $n \neq 0$ and $R(0) =1$ in Theorem~\ref{shap} above yields
$$
\sum_{n \geq 1} r(n) \frac{8n^2+4n+1}{2n(2n+1)(4n+1)} = 1.
$$
\end{example}

\begin{example}
Taking $R$ defined by $R(n) = \log n - \log(n+1)$ for $n \neq 0$ and $R(0) = 0$ 
in Theorem~\ref{shap} above yields
$$
\sum_{n \geq 1} r(n) \log\frac{(2n+1)^4}{(n+1)^2(4n+1)^2} = - \log 2.
$$
Hence
$$
\sum_{n \geq 0} r(n) \log\frac{(2n+1)^2}{(n+1)(4n+1)} = - \frac{1}{2} \log 2.
$$
After exponentiating we obtain:
$$
\prod_{n \geq 0} \left(\frac{(2n+1)^2}{(n+1)(4n+1)}\right)^{r(n)} = \frac{1}{\sqrt{2}}
$$
thus recovering the value of an infinite product obtained in \cite[Theorem 2, p.\ 148]{ACMS} 
(also see \cite{AllSha-jlms}).
\end{example}

\end{document}